\documentclass[twoside]{aiml18}

\usepackage{aiml18macro}

\usepackage{graphicx}
\usepackage{amsmath}
\usepackage{amssymb}
\usepackage{tikz}
\usepackage{stmaryrd}
\usepackage{hyperref}

\newcommand{\KTB}{\mathbf{KTB}}
\newcommand{\vKTB}{\mathcal{B}}

\newcommand{\deq}{\mathrel{\mathop:}=}
\newcommand{\pw}{\raise0.545ex\hbox{{\scalebox{1.0925}[1.0925]{\(\wp\)}}}}

\def\lastname{Koussas, Kowalski, Miyazaki and Stevens}

\begin{document}

\begin{frontmatter}
  \title{Normal extensions of \(\mathbf{KTB}\) of codimension 3}
  \author{James Koussas}
  \address{Department of Mathematics and Statistics \\ La Trobe University \\
    Melbourne, Australia}
  \author{Tomasz Kowalski}
  \address{Department of Mathematics and Statistics \\ La Trobe University \\
    Melbourne, Australia}  
  \author{Yutaka Miyazaki}
  \address{Osaka University of Economics and Law \\ Osaka, Japan}
  \author{Michael Stevens}
  \address{Research School of Information Sciences and Engineering \\ Australian
    National University \\ Canberra, Australia}

\begin{abstract}
It is known that in the lattice of normal extensions of the
logic \(\mathbf{KTB}\) there are unique logics of codimensions
\(1\) and \(2\), namely, the logic of a single reflexive point, and
the logic of the total relation on two points. A natural question arises about
the cardinality of the set of normal extensions of \(\mathbf{KTB}\) of
codimension \(3\). Generalising two finite examples
found by a computer search, we construct an uncountable family of (countable)
graphs, and prove that certain frames based on these produce 
a continuum of normal extensions of \(\mathbf{KTB}\) of codimension \(3\). We use
algebraic methods, which in this case turn out to be better suited to the task
than frame-theoretic ones. 
\end{abstract}

\end{frontmatter}

\section{Introduction}

The Kripke semantics of \(\KTB\) is the class of reflexive and symmetric frames,
that is, frames  whose accessibility relation is a \emph{tolerance}. Since
irreflexivity is not modally definable, it can be argued that \(\KTB\) is
\emph{the} logic of simple graphs. Yet \(\KTB\) is much less investigated that
its transitive cousins, and in fact certain tools working very well for
transitive logics (for example, canonical formulas) have no \(\KTB\) counterparts
working nearly as well. Among the articles dealing specifically
with \(\KTB\) and its extensions, Kripke incompleteness in various guises was investigated
in~\cite{Miya07c} and~\cite{Kos08}, interpolation in~\cite{Kos12} and~\cite{Kos16b},
normal forms in~\cite{Miya07a}, and splittings in~\cite{Miya07b}, \cite{KM09}
and~\cite{Kos16a}. In the present article we focus on the upper part of the
lattice of normal (axiomatic) extensions of \(\KTB\), or viewed dually,
the lower part of the lattice of subvarieties of the corresponding 
variety of modal algebras. 

The article is centred around a single construction, so it is structured rather
simply: in the present section we give necessary preliminaries, in Section~2 we
outline the history of the problem, in Section~3 we present the main
construction and in Section~4 we draw 
the conclusion that there are uncountably many extensions of \(\KTB\) of
codimension 3.  

Although we will use algebraic methods, we wish to 
move rather freely between graphs, frames and algebras. To make these
transitions smooth we now establish a few conventions, the general principle
behind them being that italic capitals stand for graphs, blackboard bold
capitals for Kripke frames, and boldface capitals for algebras.
With every simple graph \(G =\langle V; E\rangle \), finite or infinite, we
associate a Kripke frame 
\(\mathbb{G}\) with the same universe and the reflexive closure of \(E\) as the
accessibility relation.  For example, \(\mathbb{K}_i\) will be a
looped version of \(K_i\), the complete graph on \(i\) vertices. Thus,
\(\mathbb{K}_1\) is a single reflexive point, and \(\mathbb{K}_2\) a two-element
cluster. We will refer to these frames simply as graphs, unless
the context calls for disambiguation.
For a graph \(\mathbb{G}\), we will write \(\mathsf{Cm}(\mathbb{G})\), to
denote its complex algebra. The figure below illustrates our conventions.

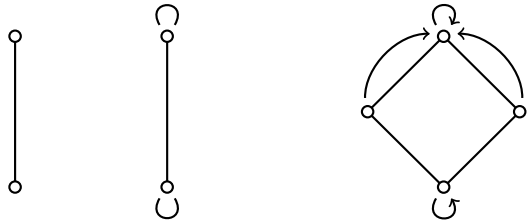
\begin{figure}[h]
\centering
\begin{tikzpicture}
\draw[thick] (-2,0) -- (-2,2);
\draw[fill=white, thick] (-2,2) circle (0.075);
\draw[fill=white, thick] (-2,0) circle (0.075);
\draw[thick] (0,0) -- (0,2);
\draw[fill=white, thick] (0,2) circle (0.075);
\draw[fill=white, thick] (0,0) circle (0.075);
\draw[thick] (-0.1,2.15) .. controls (-0.3,2.5) and (0.3,2.5) .. (0.1,2.15);
\draw[thick] (-0.1,-0.15) .. controls (-0.3,-0.5) and (0.3,-0.5) .. (0.1,-0.15);
\end{tikzpicture} \hspace{2cm}
\begin{tikzpicture}
\draw[thick] (1,0) -- (0,1) -- (-1,0) -- (0,-1) -- cycle;
\draw[fill=white, thick] (0,1) circle (0.075);
\draw[fill=white, thick] (0,-1) circle (0.075);
\draw[fill=white, thick] (1,0) circle (0.075);
\draw[fill=white, thick] (-1,0) circle (0.075);
\draw[thick, -to] (-0.1,1.15) .. controls (-0.3,1.5) and (0.3,1.5) .. (0.1,1.15);
\draw[thick, -to] (-0.1,-1.15) .. controls (-0.3,-1.5) and (0.3,-1.5) .. (0.1,-1.15);
\draw[thick,-to] (10:1.05) .. controls (30:1.2) and (60:1.2) .. (80:1.05);
\draw[thick,-to] (170:1.05) .. controls (150:1.2) and (120:1.2) .. (100:1.05);
\end{tikzpicture}
\caption{Diagrams of \(K_2\), \(\mathbb{K}_2\) and \(\mathsf{Cm}(\mathbb{K}_2)\).}
\label{k2}
\end{figure}

 If \(\mathbb{G}\) is infinite,
\(\mathsf{Cm}(\mathbb{G})\) will 
typically be too big for our purposes, but certain special subalgebras of
\(\mathsf{Cm}(\mathbb{G})\) will play a critical 
role. These algebras are mathematically the same as general (descriptive) frames over
\(\mathbb{G}\), so the machinery of bounded morphisms reduces
in these cases to verifying whether the identity map is one. The identity
map is of course frame-theoretically invisible, so all that remains is algebra.
This is essentially why algebraic methods are better suited to the
task.

We assume familiarity with the basics of universal algebra and model theory.
To be more precise, ultraproducts and {\L}o\'s Theorem, J\'onsson's Lemma for
congruence-distributive varieties, and some consequences of the congruence extension
property will suffice. All of these concepts are covered in \cite{CBUA} and \cite{BS81}.  
Our algebraic notation is standard: we use upright \(\mathop{\textup{I}}\),
\(\mathop{\textup{H}}\), \(\mathop{\textup{S}}\), \(\mathop{\textup{P}}\),
and \(\mathop{\textup{P}_\textup{U}} \) for the usual class operators of
taking isomorphic copies, homomorphic images, subalgebras, direct products and
ultraproducts, respectively. We also write \(\mathop{\textup{Si}}(\mathcal{C})\)
for the class of subdirectly irreducible algebras in \(\mathcal{C}\). 
The variety generated by a class of algebras \(\mathcal{C}\) we denote by
\(\mathop{\textup{Var}}(\mathcal{C})\), so 
\(\mathop{\textup{Var}}\) is a shorthand for \(\mathop{\textup{HSP}}\).
When we deal with Boolean algebras of sets, we use the standard set theoretical
\(\cup\) and \(\cap\), and we write \({\sim}X\) instead of \(\neg X\) for the complement of
\(X\).

\subsection{\(\KTB\)-algebras}

A \emph{\(\KTB\)-algebra} is an algebraic structure 
\(\mathbf{A}=\langle A;\vee,\wedge,\neg,\Diamond,0,1\rangle\) such that
\(\langle A;\vee,\wedge,\neg,0,1\rangle\) is a Boolean algebra,
and \(\Diamond\) a unary operation satisfying the following conditions:
\begin{enumerate}
\item \(\Diamond 0=0\),
\item \(\Diamond(x\vee y) = \Diamond x\vee  \Diamond y\),
\item \(x\leqslant  \Diamond x\),
\item \(x\leqslant \Box\Diamond x\),  
\end{enumerate}
where \(\Box\), as usual, stands for \(\neg\Diamond\neg\)
The last two conditions can be rendered as identities and so the class of
\(\KTB\)-algebras is a variety, which we will denote by \(\vKTB\).
The inequality (iv) is also equivalent to: 
\[
x\wedge \Diamond y = 0 \iff \Diamond x\wedge y = 0. 
\]
Therefore, \(\Diamond\) is a \emph{self-conjugate operator} in the sense of
\cite{JT51}, \cite{JT52} and so \(\vKTB\) is a variety of self-conjugate Boolean
Algebras with  Operators (BAOs). Incidentally, the equational axiomatisation
above is equivalent to the quasiequational one below:
\begin{enumerate}
\item[(1)] \(x\leqslant y \Longrightarrow \Diamond x\leqslant \Diamond y\),
\item[(2)] \(x\leqslant  \Diamond x\),
\item[(3)] \(x\leqslant \Box\Diamond x\). 
\end{enumerate}

For completeness, we include the following well known
propositions (see \cite{JT51}, \cite{TTML}, \cite{CBUA} and \cite{BS81} for proofs and useful exercises). The first two deal with \(\KTB\)-algebras, and the third one
recalls some crucial facts from universal algebra. 

\begin{proposition}\label{B-gen} 
For any graph \(G = \langle V; E\rangle \), the algebra \(\mathsf{Cm}(\mathbb{G})\) is a
\(\KTB\)-algebra. The class of all such algebras generates the variety \(\vKTB\). 
\end{proposition}  

\begin{proposition}\label{CEP-and-CD}
The variety \(\mathcal{B}\) is congruence distributive and has the congruence
extension property.
\end{proposition}

\begin{proposition}\label{ua-basics}
Let \(\mathcal{V}\) be a variety of algebras, and \(\mathcal{C}\) a subclass of
\(\mathcal{V}\). 
\begin{enumerate}
\item
If \(\mathcal{V}\) has the congruence extension property, \(\mathbf{A} \) is a simple algebra in \(\mathcal{V}\) and \(\mathbf{B} \in \mathop{\textup{IS}}(\mathbf{A})\), then \(\mathbf{B}\) is simple.
\item
 If \(\mathcal{V}\) has the congruence extension property, then
\(\mathop{\textup{HS}}(\mathcal{C}) = \mathop{\textup{SH}}(\mathcal{C})\). 
\item 
If \(\mathcal{V}\) is congruence distributive, then 
  \(\mathop{\textup{Si}}(\mathop{\textup{Var}}(\mathcal{C}))
  = \mathop{\textup{Si}}(\mathop{\textup{HSP}_\textup{U}}(\mathcal{C}))\).
\item
We have \(\mathcal{V} = \mathop{\textup{Var}}(\mathop{\textup{Si}}(\mathcal{V}))\).
\end{enumerate}  
\end{proposition}  

As usual, we define the term operations \(\Diamond^n\), one for each \(n\),
recursively, putting 
\(\Diamond^0x = x\) and \(\Diamond^{n+1}x = \Diamond\Diamond^nx\).

\begin{definition}\label{nclo}
Let \(\textup{\textbf{B}} = \langle B; \vee, \wedge, \neg, \Diamond,0,1 \rangle
\in \mathcal{B}\). Then the map \(\gamma\colon B\to B\) given by \(\gamma(x) =
\Box\Diamond x\) is a closure operator on \(\mathbf{B}\), which we call the
\emph{natural closure   operator} on \(\textup{\textbf{B}}\). 
\end{definition}

The following properties of natural closure operators will be useful.

\begin{lemma}\label{natclo}
Let \(\textup{\textbf{B}} = \langle B; \vee, \wedge , \neg, \Diamond, 0 ,1
\rangle \in \mathcal{B}\) and let \(\gamma\) denote the natural closure operator
on \(\textup{\textbf{B}}\). 
\begin{enumerate}
\item[\textup{(i)}]
If \(x\in B\) is \(\gamma\)-closed, then \(\neg x = \Diamond\neg\Diamond x\) and
\(\Diamond\neg x = \Diamond^2\neg\Diamond x\). 
\item[\textup{(ii)}]
If \(x\in B\), then \(\Diamond\gamma(x) = \Diamond x\).
\end{enumerate}
\end{lemma}

\begin{proof}
Let \(x\in B\). If \(x\) is \(\gamma\)-closed, then \(x = \Box\Diamond x\),
thus \(\neg x = \Diamond\neg\Diamond x\) and so
\(\Diamond\neg x = \Diamond^2\neg\Diamond x\), hence (i) holds.  

As \(\gamma\) is a closure operator, we have  \(x \leqslant \gamma(x)\), hence
\(\Diamond x \leqslant \Diamond \gamma(x)\). Similarly, \(\neg\Diamond x
\leqslant \gamma(\neg\Diamond x) = \Box\Diamond\neg\Diamond x =
\neg\Diamond\Box\Diamond x = \neg\Diamond\gamma(x)\),
so \(\Diamond \gamma(x)\leqslant \Diamond x\). Thus,
\(\Diamond \gamma(x) = \Diamond x\), hence (ii) holds.
\end{proof} 

\begin{lemma}\label{nicegen}
Let \(\textup{\textbf{B}} = \langle B; \vee, \wedge, \neg, \Diamond , 0 ,1\rangle\in
\mathcal{B}\) and  let \(\gamma\) be the natural closure operator of \(\textup{\textbf{B}}\). If \(\textup{\textbf{B}} \models \exists x \colon x \neq 0 \mathrel{\&} \Diamond x \neq 1\) and
\(\textup{\textbf{B}} \models \forall x \colon x \neq 0 \to \Diamond^n x = 1\), for some \(n\in \omega\setminus \{0\}\), then there is a \(\gamma\)-closed \(y \in B\) with
\(\Diamond y \neq 1\) and \(\Diamond^2 y =1 \). 
\end{lemma}

\begin{proof}
Let \(x\) be a witness of \(\exists x \colon x \neq 0 \mathrel{\&} \Diamond x
\neq 1\) in \(\mathbf{B}\). By assumption, \(\textbf{B} \models \forall x \colon
x \neq 0 \to \Diamond^n x = 1\), so we must have \(\Diamond^n x = 1\). Hence,
there is some \(m \in\{1,\dots,n-1\}\) with \(\Diamond^{m} x \neq 1\) and
\(\Diamond^{m+1} x=1\).  By Lemma
\ref{natclo}(ii), \(\Diamond \gamma(\Diamond^{m-1} x) = \Diamond^m  x \neq 1\)
and \(\Diamond^{m+1} x = 1\). Since \(\gamma(\Diamond^{m-1}x)\)
is \(\gamma\)-closed, putting \(y = \gamma(\Diamond^{m-1}(x))\), we get a
\(\gamma\)-closed \(y \in B\) with \(\Diamond y \neq 1\) and \(\Diamond^2 y =1
\), as required. 
\end{proof}

\section{The history of the problem}

A logic \(L\) is said to have codimension \(n\), in some lattice \(\Lambda\) of
logics, if there exists a descending chain 
\(L_0\succ\dots \succ L_n\) of logics from \(\Lambda\), such that \(L_0\) is
inconsistent, \(L_n = L\), and \(L_{i-1}\) covers \(L_i\) for each
\( i\in\{0,\dots,n\}\).  Lattices of nonclassical logics are typically very
complicated, so looking at logics of small codimensions is one way of analysing
these lattices. In particular, finding the smallest \(n\) for which
there are uncountably many logics of codimension \(n\) in \(\Lambda\) indicates
at which level the lattice gets really badly complicated.

Let \(\mathrm{NExt}(\KTB)\) stand for the lattice of normal extensions
of \(\mathbf{KTB}\), where we identify logics with their sets of theorems.
We intend to show that for \(\Lambda = \mathrm{NExt}(\KTB)\) the smallest such \(n\) is \(3\).  

\begin{remark}\label{quasiv}
If we identified logics with their \emph{consequence operations}, rather than
their sets of theorems, \(\mathrm{NExt}(\KTB)\) would be the 
the lattice of normal \emph{axiomatic} extensions of
\(\mathbf{KTB}\). Let us call the lattice of all normal extensions of \(\mathbf{KTB}\),
whether axiomatic or not, \(\mathrm{CNExt}(\KTB)\). Then  
\(\mathrm{NExt}(\KTB)\) is a subposet of \(\mathrm{CNExt}(\KTB)\).
However, the codimension of a logic \(L\in\mathrm{NExt}(\KTB)\) can 
be smaller in \(\mathrm{NExt}(\KTB)\) than the codimension of \(L\) in
\(\mathrm{CNExt}(\KTB)\). It follows from results of
Blanco, Campercholi and Vaggione (see Theorem~1 in~\cite{BCV01}) that
for any logic \(L\in\mathrm{NExt}(\KTB)\) of codimension at least 2,
\(\mathrm{NExt}(L)\) is \emph{strictly} contained in \(\mathrm{CNExt}(L)\).
\end{remark}

Let \(\mathrm{Subv}(\vKTB)\) stand for the lattice of subvarieties of
\(\mathcal{B}\). Then, 
the usual dual isomorphism between \(\mathrm{NExt}(\KTB)\) and
\(\mathrm{Subv}(\vKTB)\) holds, and therefore logics of
codimension \(n\)  in 
\(\mathrm{NExt}(\KTB)\) correspond to varieties of height \(n\) in
\(\mathrm{Subv}(\vKTB)\). The next theorem gives a complete picture of
\(\mathrm{Subv}(\vKTB)\) up to height 2, and therefore, dually, of
\(\mathrm{NExt}(\KTB)\) down to codimension 2. The second statement in the theorem
is due to the third author (see~\cite{Miya07b}). 
 
\begin{thm}\label{known}
The lattice \(\mathrm{Subv}(\vKTB)\) has exactly one atom, namely
\(\mathop{\textup{Var}}(\mathsf{Cm}(\mathbb{K}_1))\). This atom in turn has
exactly one cover, namely \(\mathop{\textup{Var}}(\mathsf{Cm}(\mathbb{K}_2))\).  
\end{thm}

A natural question then arises about the cardinality of the ``set'' of varieties covering 
\(\mathop{\textup{Var}}(\mathsf{Cm}(\mathbb{K}_2))\). It is easy to show that
this ``set'' is infinite: countably 
many varieties covering \(\mathop{\textup{Var}}(\mathsf{Cm}(\mathbb{K}_2))\)
were constructed by the second and fourth author in an unpublished note~\cite{KS05},
using certain finite graphs. But finite graphs clearly could not suffice for
a construction of uncountably many varieties covering 
\(\mathop{\textup{Var}}(\mathsf{Cm}(\mathbb{K}_2))\). A construction of an
appropriate uncountable family 
of countably infinite graphs began by finding two finite ones, called below
\(\mathbb{G}_1\) and \(\mathbb{G}_2\): 

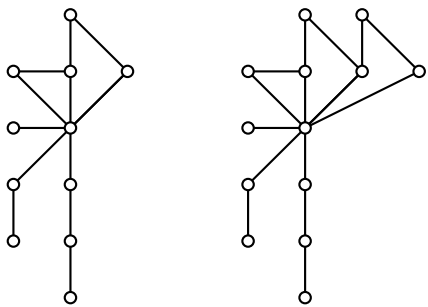
\begin{figure}[h]
\centering
\begin{tikzpicture}[scale=0.75]
\draw[thick] (0,0) -- (-1,0);
\draw[thick] (0,0) -- (1,1);
\draw[thick] (0,0) -- (-1,-1) -- (-1,-2);
\draw[thick] (1,1) -- (0,2) -- (0,-3);
\draw[thick] (0,1) -- (-1,1) -- (0,0) -- (1,1);
\draw[thick, fill=white] (-1,1) circle (0.1);
\draw[thick, fill=white] (0,1) circle (0.1);
\draw[thick, fill=white] (1,1) circle (0.1);
\draw[thick, fill=white] (0,2) circle (0.1);
\draw[thick, fill=white] (0,0) circle (0.1);
\draw[thick, fill=white] (0,-1) circle (0.1);
\draw[thick, fill=white] (0,-2) circle (0.1);
\draw[thick, fill=white] (0,-3) circle (0.1);
\draw[thick, fill=white] (-1,-1) circle (0.1);
\draw[thick, fill=white] (-1,-2) circle (0.1);
\draw[thick, fill=white] (-1,0) circle (0.1);
\end{tikzpicture}\qquad\qquad\begin{tikzpicture}[scale=0.75]
\draw[thick] (0,0) -- (-1,0);
\draw[thick] (0,0) -- (1,1);
\draw[thick] (0,0) -- (2,1);
\draw[thick] (0,0) -- (-1,-1) -- (-1,-2);
\draw[thick] (1,1) -- (0,2) -- (0,-3);
\draw[thick] (0,1) -- (-1,1) -- (0,0) -- (1,1) -- (1,2) -- (2,1);
\draw[thick, fill=white] (-1,1) circle (0.1);
\draw[thick, fill=white] (0,1) circle (0.1);
\draw[thick, fill=white] (1,1) circle (0.1);
\draw[thick, fill=white] (2,1) circle (0.1);
\draw[thick, fill=white] (0,2) circle (0.1);
\draw[thick, fill=white] (1,2) circle (0.1);
\draw[thick, fill=white] (0,0) circle (0.1);
\draw[thick, fill=white] (0,-1) circle (0.1);
\draw[thick, fill=white] (0,-2) circle (0.1);
\draw[thick, fill=white] (0,-3) circle (0.1);
\draw[thick, fill=white] (-1,-1) circle (0.1);
\draw[thick, fill=white] (-1,-2) circle (0.1);
\draw[thick, fill=white] (-1,0) circle (0.1);
\end{tikzpicture}\label{nauty-graphs}
\caption{Graph drawings of \(\mathbb{G}_1\) and \(\mathbb{G}_2\) (with loops omitted).}
\end{figure}

These were found by the second and fourth authors through a computer
search, performed with the help of Brendan McKay's \texttt{nauty}
(see~\cite{nauty}). All non-isomorphic graphs with up to 13 vertices were generated,
and checked for the property of not admitting any bounded morphism, except
the identity map, the constant map onto $\mathbb{K}_1$, and a bounded morphism onto
$\mathbb{K}_2$. By finiteness, this is sufficient (and also necessary) for the logic
of such a graph $\mathbb{G}$ to be of codimension 3, or, equivalently, for
\(\mathop{\textup{Var}}(\mathsf{Cm}(\mathbb{G}))\) to be a cover of  
\(\mathop{\textup{Var}}(\mathsf{Cm}(\mathbb{K}_2))\).

Two of these graphs are depicted in
Figure~\ref{nauty-graphs}. They were the only ones that revealed a workable
family resemblance to one another.
They were also so different from the finite graphs considered
in~\cite{KS05} as to be completely unexpected to the finders. 
Verifying by hand that the bounded morphism condition mentioned above indeed
holds, is tedious but not difficult, and so it was proved that
\(\mathop{\textup{Var}}(\mathsf{Cm}(\mathbb{G}_1))\) and
\(\mathop{\textup{Var}}(\mathsf{Cm}(\mathbb{G}_2))\) indeed cover   
\(\mathop{\textup{Var}}(\mathsf{Cm}(\mathbb{K}_2))\), confirming the
computer-assisted finding.

Extending the zigzaging pattern infinitely to the right is then a no-brainer, and
a suitable twisting of the zigzag produces an uncountable family of pairwise
non-isomorphic  graphs. The next step is to take certain subalgebras of the
complex algebras of 
these infinite graphs (unlike in the finite case, the full complex algebras
may not do), and prove that the varieties they generate are pairwise
distinct and cover \(\mathop{\textup{Var}}(\mathsf{Cm}(\mathbb{K}_2))\). The
three last authors did 
produce a rough approximation to a proof, which was convincing enough (for them)
to announce the result (see~\cite{KMS07}). However, the full
proof was never published, 
and in fact it did not exist, as the details were never satisfactorily verified.
The three authors dispersed around the globe and the proof was left unfinished.
It took about 10 years, and the first author, to produce a complete proof.
We are going to present it now.

\section{Construction}

Before we begin, we make one more remark on the methods. The construction presented 
below may at first glance suggest that the reasoning about ultrapowers, which
will play an important part in the proofs, is not necessary, because everything
that could go wrong in an ultrapower already goes wrong in the original
algebra. Were it so, the proofs could be greatly simplified, but unfortunately
the first glance is misleading. There exists an infinite 
\(\KTB\)-algebra $\mathbf{A}$ such that $\textup{HS}(\mathbf{A})$
does not contain \(\mathsf{Cm}(\mathbb{K}_3)\), but
$\textup{HSP}_\textup{U}(\mathbf{A})$ does, so $\mathbf{A}$ does not
generate a cover of $\textup{Var}(\mathbb{K}_2)$. Considering ultrapowers is
necessary, at least in principle.

Now, for the construction. Firstly, we will need the following Lemma, which is
an easy consequence of Proposition~\ref{ua-basics}(iii).  

\begin{lemma}\label{k2si}
We have \(\mathop{\textup{Si}}(\mathop{\textup{Var}}(\mathsf{Cm}(\mathbb{K}_2)))
= \mathop{\textup{I}}(\{\mathsf{Cm}(\mathbb{K}_1),
\mathsf{Cm}(\mathbb{K}_2)\})\). 
\end{lemma}

Next, we state a sufficient set of conditions for an algebra in
\(\vKTB\) to generate a variety of height 3.  

\begin{lemma}\label{conditions}
Let \(\mathbf{A} \in \mathcal{B}\) and assume that \(\mathbf{A}\) has the
following properties: 
\begin{enumerate}
\item  \(\mathbf{A}\) is infinite;
\item
 \(\mathsf{Cm}(\mathbb{K}_2) \in \mathop{\textup{IS}} (\mathbf{A})\);
\item
every member of \(\mathop{\textup{P}}_\textup{U}(\mathbf{A})\) is simple;
\item
for all \(\mathbf{B} \in \mathop{\textup{ISP}}_\textup{U} (\mathbf{A})\), we have \(\mathbf{B} \cong \mathsf{Cm}(\mathbb{K}_1)\), \(\mathbf{B} \cong \mathsf{Cm}(\mathbb{K}_2)\) or \(\mathbf{A} \in \mathop{\textup{IS}} (\mathbf{B} )\).
\end{enumerate}
Then \(\mathop{\textup{Var}}(\mathbf{A})\) is of height 3. 
\end{lemma}

\begin{proof}
Based on (iii), \(\mathop{\textup{HP}}_\mathrm{U}(\mathbf{A}) = \mathop{\textup{I}}(\{ \mathbf{T}\} \cup\mathop{\textup{P}}_\textup{U}(\mathbf{A})) \), for some trivial \(\mathbf{T} \in \mathcal{B}\). So, by Proposition \ref{ua-basics}, \(\mathop{\textup{Si}}(\mathop{\textup{Var}}(\mathbf{A})) = \mathop{\textup{Si}}(\mathop{\textup{HSP}}_\mathrm{U}(\mathbf{A})) =\mathop{\textup{Si}}(\mathop{\textup{SHP}}_\mathrm{U}(\mathbf{A}))  = \mathop{\textup{ISP}}_\mathrm{U}(\mathbf{A})\). Clearly, \(\mathbf{A} \in \mathop{\textup{ISP}}_\mathrm{U}(\mathbf{A})\), so (i), (ii) and Lemma \ref{k2si} tell us that \(\mathop{\textup{Var}}(\mathbf{A})\) properly extends \(\mathop{\textup{Var}}(\mathsf{Cm}(\mathbb{K}_2))\).  Let \(\mathcal{V}\) be a variety with \(\mathop{\textup{Var}}(\mathsf{Cm}(\mathbb{K}_2)) \subseteq \mathcal{V} \subseteq \mathop{\textup{Var}}(\mathbf{A})\). Since \(\mathcal{V} \subseteq \mathop{\textup{Var}}(\mathbf{A})\), we have \(\mathop{\textup{Si}}(\mathcal{V}) \subseteq \mathop{\textup{Si}}(\mathop{\textup{Var}}(\mathbf{A})) = \mathop{\textup{ISP}}_\mathrm{U}(\mathbf{A})\). Combining this with (iv) and Lemma \ref{k2si}, we f\-ind that \(\mathop{\textup{Si}}(\mathcal{V}) \subseteq \mathop{\textup{Si}}(\mathsf{Cm}(\mathbb{K}_2)) \) or \(\mathbf{A} \in \mathop{\textup{IS}}(\mathcal{V}) = \mathcal{V}\). So, by Proposition \ref{ua-basics}, we must have \(\mathcal{V} = \mathop{\textup{Var}}(\mathsf{Cm}(\mathbb{K}_2))\) or \(\mathcal{V} = \mathop{\textup{Var}}(\mathbf{A})\). Hence, \(\mathop{\textup{Var}}(\mathbf{A})\) covers \(\mathop{\textup{Var}}(\mathsf{Cm}(\mathbb{K}_2))\), so \(\mathop{\textup{Var}}(\mathbf{A})\) has height 3, as claimed. \end{proof}

Our construction of a continuum of subvarieties of \(\vKTB\) of
height 3 begins with the following definition.

\begin{definition}\label{N-graph}
Let \(\mathbb{E}\) denote the set of positive even numbers, let \(A = \{ a\}\),
\(B = \{b_1, b_2, b_3\}\), \(C = \{c_1, c_2\}\), \(D = \{d \}\), \(U = \{u_i
\mid i \in \omega \setminus \{0\}\}\) and \(L = \{\ell_i \mid  i\in \omega \}\)
be pairwise disjoint, and assume that \(u_i \neq u_j\), \(\ell_i \neq \ell_j\),
\( b_i \neq b_j\) and \(c_i \neq c_j\) whenever \(i \neq j\). Now, define \(U_i
\deq \{u_i\}\), for all \(i\in \omega \setminus \{0\}\), \(L_i \deq
\{\ell_i\}\),  for all \(i \in \omega\), \(B_i \deq \{b_i\}\), for all \(i \in
\{1,2,3\}\), \(C_i \deq \{c_i\}\), for all \(i \in \{1,2\}\), and  \(P \deq \{
b_1, c_1, d\}\). For each \(N \subseteq \mathbb{E}\), let \(\mathbb{F}_N\) be
the graph \(\langle W; R_N \rangle\), where \(W \deq A \cup B \cup C \cup D \cup
U \cup L\) and \(R_N\) is the relation def\-ined by 
\[
  x \mathrel{R_N} y \iff x=y \textrm{ or } \{x,y\} = \begin{cases}   \{a, b_i\},
    \textrm{ for some } i \in \{1,2,3\},\\
 \{b_i, c_i\}, \textrm{ for some } i \in \{1,2\},\\
 \{c_1, d\},\\
 \{\ell_0, \ell_1\}, \\ 
 \{a, \ell_i \}, \textrm{ for some } i\in \omega, \\
 \{\ell_i, u_i \}  , \textrm{ for some } i \in \omega \setminus \{0\}, \\
 \{\ell_i, u_{i-1}\}, \textrm{ for some } i\in \mathbb{E}, \\
 \{\ell_i, u_{i+1}\}, \textrm{ for some } i \in N \textrm{ or } \\ 
 \{\ell_{i+1}, u_i\}, \textrm{ for some } i \in \mathbb{E} \setminus N.\end{cases}
\]
\end{definition}

As usual with graphs, a picture is worth a thousand words. Certainly it is worth
all the words of the definition above.  Here it is.

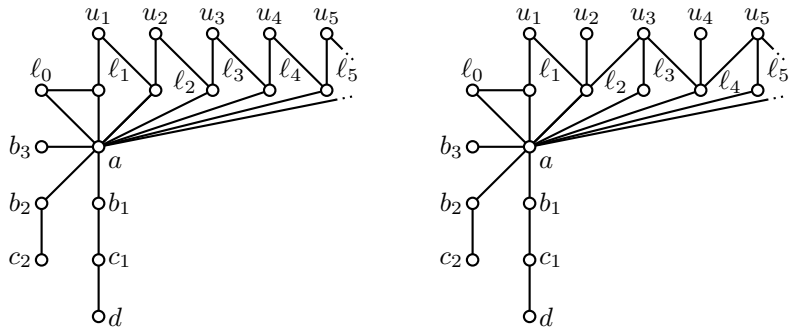
\begin{figure}[h]
\centering
\begin{tikzpicture}[scale=0.75]
\draw[thick, dotted] (4.25,1.75) -- (4.5,1.5);
\draw[thick, dotted] (4.15,0.83) -- (4.5,0.9);
\draw[thick] (4.16,0.832) -- (0,0);
\draw[thick] (4.255,1.745) -- (4,2);
\draw[thick] (0,0) -- (-1,0);
\draw[thick] (0,0) -- (1,1);
\draw[thick] (0,0) -- (2,1);
\draw[thick] (0,0) -- (3,1);
\draw[thick] (0,0) -- (4,1);
\draw[thick] (0,0) -- (-1,-1) -- (-1,-2);
\draw[thick] (1,1) -- (0,2) -- (0,-3);
\draw[thick] (0,1) -- (-1,1) -- (0,0) -- (1,1) -- (1,2) -- (2,1) -- (2,2) -- (3,1) -- (3,2);
\draw[thick] (3,2) -- (4,1) -- (4,2);
\draw[thick, fill=white] (-1,1) circle (0.1);
\draw[thick, fill=white] (0,1) circle (0.1);
\draw[thick, fill=white] (1,1) circle (0.1);
\draw[thick, fill=white] (2,1) circle (0.1);
\draw[thick, fill=white] (3,1) circle (0.1);
\draw[thick, fill=white] (4,1) circle (0.1);
\draw[thick, fill=white] (0,2) circle (0.1);
\draw[thick, fill=white] (1,2) circle (0.1);
\draw[thick, fill=white] (2,2) circle (0.1);
\draw[thick, fill=white] (3,2) circle (0.1);
\draw[thick, fill=white] (4,2) circle (0.1);
\draw[thick, fill=white] (0,0) circle (0.1);
\draw[thick, fill=white] (0,-1) circle (0.1);
\draw[thick, fill=white] (0,-2) circle (0.1);
\draw[thick, fill=white] (0,-3) circle (0.1);
\draw[thick, fill=white] (-1,-1) circle (0.1);
\draw[thick, fill=white] (-1,-2) circle (0.1);
\draw[thick, fill=white] (-1,0) circle (0.1);
\draw[above] (0,2) node {\(u_1\)};
\draw[above] (1,2) node {\(u_2\)};
\draw[above] (2,2) node {\(u_3\)};
\draw[above] (3,2) node {\(u_4\)};
\draw[above] (4,2) node {\(u_5\)};
\draw[above] (-1,1) node {\(\ell_0\)};
\draw[above right] (0,1) node {\(\ell_1\)};
\draw[right] (1.15,1.15) node {\(\ell_2\)};
\draw[above right] (2,1) node {\(\ell_3\)};
\draw[above right] (3,1) node {\(\ell_4\)};
\draw[above right] (4,1) node {\(\ell_5\)};
\draw[below right] (0,0) node {\(a\)};
\draw[right] (0,-1) node {\(b_1\)};
\draw[right] (0,-2) node {\(c_1\)};
\draw[right] (0,-3) node {\(d\)};
\draw[left] (-1,-1) node {\(b_2\)};
\draw[left] (-1,-2) node {\(c_2\)};
\draw[left] (-1,0) node {\(b_3\)};
\end{tikzpicture} \qquad \begin{tikzpicture}[scale=0.75]
\draw[thick, dotted] (4.25,1.75) -- (4.5,1.5);
\draw[thick, dotted] (4.15,0.83) -- (4.5,0.9);
\draw[thick] (4.16,0.832) -- (0,0);
\draw[thick] (4.255,1.745) -- (4,2);
\draw[thick] (0,0) -- (-1,0);
\draw[thick] (0,0) -- (1,1);
\draw[thick] (0,0) -- (2,1);
\draw[thick] (0,0) -- (3,1);
\draw[thick] (0,0) -- (4,1);
\draw[thick] (0,0) -- (-1,-1) -- (-1,-2);
\draw[thick] (1,1) -- (0,2) -- (0,-3);
\draw[thick] (0,1) -- (-1,1) -- (0,0) -- (1,1) -- (2,2); 
\draw[thick] (1,1) -- (1,2); 
\draw[thick] (2,1) -- (2,2);
\draw[thick] (3,1) -- (3,2);
\draw[thick] (2,2) -- (3,1) -- (4,2);
\draw[thick] (4,1) -- (4,2);
\draw[thick, fill=white] (-1,1) circle (0.1);
\draw[thick, fill=white] (0,1) circle (0.1);
\draw[thick, fill=white] (1,1) circle (0.1);
\draw[thick, fill=white] (2,1) circle (0.1);
\draw[thick, fill=white] (3,1) circle (0.1);
\draw[thick, fill=white] (4,1) circle (0.1);
\draw[thick, fill=white] (0,2) circle (0.1);
\draw[thick, fill=white] (1,2) circle (0.1);
\draw[thick, fill=white] (2,2) circle (0.1);
\draw[thick, fill=white] (3,2) circle (0.1);
\draw[thick, fill=white] (4,2) circle (0.1);
\draw[thick, fill=white] (0,0) circle (0.1);
\draw[thick, fill=white] (0,-1) circle (0.1);
\draw[thick, fill=white] (0,-2) circle (0.1);
\draw[thick, fill=white] (0,-3) circle (0.1);
\draw[thick, fill=white] (-1,-1) circle (0.1);
\draw[thick, fill=white] (-1,-2) circle (0.1);
\draw[thick, fill=white] (-1,0) circle (0.1);
\draw[above] (0,2) node {\(u_1\)};
\draw[above] (1,2) node {\(u_2\)};
\draw[above] (2,2) node {\(u_3\)};
\draw[above] (3,2) node {\(u_4\)};
\draw[above] (4,2) node {\(u_5\)};
\draw[above] (-1,1) node {\(\ell_0\)};
\draw[above right] (0,1) node {\(\ell_1\)};
\draw[right] (1.15,1.15) node {\(\ell_2\)};
\draw[above right] (2,1) node {\(\ell_3\)};
\draw[right] (3.15,1.15) node {\(\ell_4\)};
\draw[above right] (4,1) node {\(\ell_5\)};
\draw[below right] (0,0) node {\(a\)};
\draw[right] (0,-1) node {\(b_1\)};
\draw[right] (0,-2) node {\(c_1\)};
\draw[right] (0,-3) node {\(d\)};
\draw[left] (-1,-1) node {\(b_2\)};
\draw[left] (-1,-2) node {\(c_2\)};
\draw[left] (-1,0) node {\(b_3\)};
\end{tikzpicture}
\caption{Graph drawings of (finite sections of) \(\mathbb{F}_\varnothing\) and
  \(\mathbb{F}_{\{2,4\}}\) (with loops omitted).} 
\label{fig1}
\end{figure} 
Accordingly, in the proofs, we will frequently refer to Fig.~\ref{fig1}, as well as to
Fig.~\ref{fig2} below, rather than to Definition~\ref{N-graph}. 
Next, we define the algebras essential to our construction. The notation is
as in Definition~\ref{N-graph}.

\begin{definition}\label{N-alg}
For each \(N \subseteq \mathbb{E}\), let \(\mathbf{D}_N\) be the
subalgebra of \(\textsf{Cm}(\mathbb{F}_N)\) generated by \(D\) and let \(D_N\)
be the universe of \(\mathbf{D}_N\).
\end{definition}

From now on, we will use \(\Diamond_N\) to
stand for \(R_N^{-1}\), and we will omit the subscript \(N\) if there is no danger
of confusion.

\begin{lemma}\label{singletons}
Let \(N \subseteq \mathbb{E}\). Then \(A, B_1, B_3, C_1, D, L_i, U_j, B_2 \cup L, C_2 \cup U \in D_N\), for all \(i \in \omega\) and all \(j \in \omega \setminus \{0\}\).
\end{lemma}

\begin{proof}
By definition, \(D \in D_N\). So, based on Fig. \ref{fig1}, \(C_1 = \Diamond D  \cap {\sim} D \in D_N\). Similarly, \(B_1 = \Diamond C_1 \cap {\sim}\Diamond D \in D_N\) and \(C_2 \cup U = {\sim}\Diamond ^4 D \in D_N\), hence we have \(A = \Diamond B_1 \cap {\sim}\Diamond C_1 \in D_N\) and \(B_3 = {\sim}(\Diamond^2(C_2 \cup U) \cup \Diamond^2D) \in D_N\). From this, it follows that \(L_0 = {\sim}(B_3 \cup \Diamond(C_2 \cup U) \cup \Diamond^3 D) \in D_N\), so we have \(B_2 \cup L = (\Diamond(C_2 \cup U) \cup L_0)\cap {\sim}(C_2 \cup U) \in D_N\). Similarly, we must have \(L_1 = \Diamond L_0 \cap {\sim}(A \cup L_0) \in D_N\), which implies that \(U_1 = \Diamond L_1 \cap {\sim}\Diamond A \in D_N\).

It remains to establish that \(L_i, U_j \in D_N\),  for all \(i \in \omega\) and all \(j \in \omega \setminus \{0\}\); we proceed by induction. Assume that \(L_i, U_i \in D_N\), for some odd \(i\in \omega\).

\begin{figure}[h]
\centering
\begin{tikzpicture}
\draw[thick] (3,1) -- (3/2,0) -- (2,1);
\draw[thick] (1,1) -- (0,2) -- (0,1) -- (3/2,0)-- (1,1) -- (1,2) -- (2,1) -- (2,2) -- (3,1) -- (3,2);
\draw[thick, fill=white] (0,1) circle (0.075);
\draw[thick, fill=white] (1,1) circle (0.075);
\draw[thick, fill=white] (2,1) circle (0.075);
\draw[thick, fill=white] (0,2) circle (0.075);
\draw[thick, fill=white] (1,2) circle (0.075);
\draw[thick, fill=white] (2,2) circle (0.075);
\draw[thick, fill=white] (3/2,0) circle (0.075);
\draw[thick, fill=white] (3,2) circle (0.075);
\draw[thick, fill=white] (3,1) circle (0.075);
\draw[above] (0,2) node {\(u_i\)};
\draw[above] (1,2) node {\(u_{i+1}\)};
\draw[above] (2,2) node {\(u_{i+2}\)};
\draw[above] (3,2) node {\(u_{i+3}\)};
\draw[left] (0,1.01) node {\(\ell_{i}\)};
\draw[left] (1,1.01) node {\(\ell_{i+1}\)};
\draw[right] (2,1.01) node {\(\ell_{i+2}\)};
\draw[right] (3,1.01) node {\(\ell_{i+3}\)};
\draw[below] (3/2,0) node {\(a\)};
\end{tikzpicture}\qquad\begin{tikzpicture}
\draw[thick] (1,1) -- (0,2) -- (0,1) -- (3/2,0) -- (1,1) -- (1,2);
\draw[thick] (2,2) -- (3,1) -- (3/2,0) -- (3,1) -- (3,2);
\draw[thick] (1,1) -- (2,2) -- (2,1) -- (3/2,0); 
\draw[thick, fill=white] (0,1) circle (0.075);
\draw[thick, fill=white] (1,1) circle (0.075);
\draw[thick, fill=white] (2,1) circle (0.075);
\draw[thick, fill=white] (0,2) circle (0.075);
\draw[thick, fill=white] (1,2) circle (0.075);
\draw[thick, fill=white] (2,2) circle (0.075);
\draw[thick, fill=white] (3/2,0) circle (0.075);
\draw[thick, fill=white] (3,2) circle (0.075);
\draw[thick, fill=white] (3,1) circle (0.075);
\draw[above] (0,2) node {\(u_i\)};
\draw[above] (1,2) node {\(u_{i+1}\)};
\draw[above] (2,2) node {\(u_{i+2}\)};
\draw[above] (3,2) node {\(u_{i+3}\)};
\draw[left] (0,1.01) node {\(\ell_{i}\)};
\draw[left] (1,1.01) node {\(\ell_{i+1}\)};
\draw[right] (2,1.01) node {\(\ell_{i+2}\)};
\draw[right] (3,1.01) node {\(\ell_{i+3}\)};
\draw[below] (3/2,0) node {\(a\)};
\end{tikzpicture}
\caption{Graph drawings for Lemma \ref{singletons}.}
\label{fig2}
\end{figure}
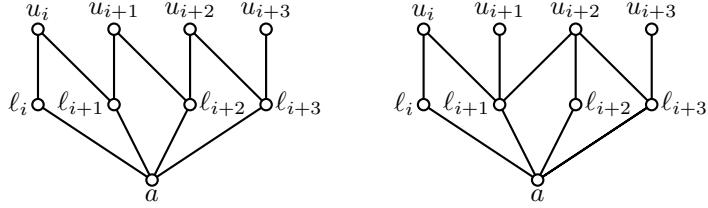 

Firstly, assume that \(i+1 \notin N\). By Fig. \ref{fig2}, \(L_{i+1} = \Diamond U_i \cap {\sim}\Diamond L_i  \in D_N\), which implies that \(U_{i+1} = \Diamond L_{i+1} \cap {\sim}(\Diamond A \cup  U_i) \in D_N\). This implies that \(L_{i+2} = \Diamond U_{i+1} \cap {\sim} \Diamond L_{i+1} \in D_N\), so  \(U_{i+2} = \Diamond L_{i+2} \cap {\sim}(\Diamond A \cup U_{i+1}) \in D_N\). Thus, \(L_{i+1}, L_{i+2}, U_{i+1}, U_{i+2} \in D_N\) if \(i+1 \notin N\).

Next, assume that \(i+1 \in N\). From Fig. \ref{fig2}, \(L_{i+1} = \Diamond U_i \cap {\sim} \Diamond L_i \in D_N\), so we have \(U_{i+1} \cup U_{i+2} =  \Diamond L_{i+1} \cap {\sim}(\Diamond A \cup U_i) \in D_N\). Using these results, we find that we must have \(L_{i+2} \cup L_{i+3} = \Diamond(U_{i+1} \cup U_{i+2}) \cap {\sim}\Diamond L_{i+1} \in D_N\). From this, it follows that \(U_{i+2} = (U_{i+1} \cup U_{i+2}) \cap \Diamond(L_{i+2} \cup L_{i+3}) \in D_N\), which implies that \(U_{i+1} = (U_{i+1} \cup U_{i+2}) \cap {\sim} U_{i+2} \in D_N\). From these results,  \(X \deq \Diamond(L_{i+2} \cup L_{i+3}) \cap {\sim}(\Diamond A \cup U_{i+2}) \in D_N\). Based on Fig \ref{fig2}, we must have \(u_{i+3}\in X\) and \(a,\ell_{i+2}, u_{i+2} \notin X\), hence \(\ell_{i+2}\notin \Diamond X\) and \(\ell_{i+3} \in \Diamond X\). Thus, \(L_{i+2} = (L_{i+2} \cup L_{i+3}) \cap {\sim}\Diamond X \in D_N\), so \(L_{i+1}, L_{i+2}, U_{i+1}, U_{i+2} \in D_N\) if \(i+1 \in N\).

In every case, we have \(L_{i+1}, L_{i+2}, U_{i+1}, U_{i+2} \in D_N\). Hence, by induction,  \(U_i, L_j \in D_N\), for all \(i\in \omega\) and all \(j\in \omega \setminus \{0\}\), so we are done. \end{proof}

\begin{corollary}\label{DN-inf}
Let \(N\subseteq \mathbb{E}\). Then the algebra \(\mathbf{D}_N\) is infinite.
\end{corollary}  

\begin{lemma} \label{embed}
Let \(N \subseteq \mathbb{E}\). Then \(\textup{\textsf{Cm}}(\mathbb{K}_2) \in
\mathop{\textup{IS}}(\mathbf{D}_N)\). 
\end{lemma}

\begin{proof}
From Lemma \ref{singletons}, it follows that \(X \deq B \cup D \cup L \in D_N\). Based on Fig. \ref{k2} and Fig. \ref{fig2}, the subalgebra of \(\mathbf{D}_N\) generated by \(X\) is isomorphic to \(\mathsf{Cm}(\mathbb{K}_2)\), hence \(\mathsf{Cm}(\mathbb{K}_2) \in \mathop{\textup{IS}}(\mathbf{D}_N)\), as claimed. \end{proof}

Based on Fig \ref{fig2}, if \(N \subseteq \mathbb{E}\), then every vertex other than \(d\) is joined to \(a\) by a path of length of at most \(2\) in \(\mathbb{F}_N\), so \(\mathbf{D}_N \models\forall x \colon  x \neq 0 \to \Diamond^5 x = 1\). The following Lemma is an easy consequence of this observation and \L o\'s's Theorem. 

\begin{lemma} \label{simple}
Let \(N \subseteq \mathbb{E}\). Then every member of\/
\(\mathop{\textup{P}}_\textup{U} (\mathbf{D}_N) \) is simple. 
\end{lemma}

\begin{lemma} \label{small}
Let \(N \subseteq \mathbb{E}\), let \(F\) be an ultrafilter over a set \(I\) and let \(\mathbf{S}\) be a subalgebra of \(\mathbf{D}_N^I/F\). If \(\mathbf{S} \models \forall x \colon x \neq 0 \to \Diamond x = 1\), then \(\mathbf{S} \cong \textup{\textsf{Cm}}(\mathbb{K}_1)\) or \(\mathbf{S} \cong \textup{\textsf{Cm}}(\mathbb{K}_2)\).
\end{lemma}

\begin{proof}
Let \(S\) be the universe of \(\mathbf{S}\) and define a map \( \bar{X} \colon I \to D_N\) by \(i \mapsto X\), for each \(X \in D_N\). Suppose, for a contradiction, that there exist \(X,Y \in D_N^I\) with \(X/F, Y/F \in S \setminus \{\bar{\varnothing}/F, \bar{W}/F\}\), \(X/F\neq Y/F\) and \(X /F \neq \neg Y/F\). Clearly, we must have \(\{ i \in I \mid d \in X(i)\} \cup \{i \in I \mid d\in {\sim} X(i)\}  = I \in F\), which implies that \(\{i \in I \mid d \in X(i)\}\in F\) or \(\{i \in I \mid d \in {\sim} X(i)\}\in F\). Similarly,  \(\{i \in I \mid d \in Y(i)\}\in F\) or \(\{i \in I \mid d \in {\sim} Y(i)\}\in F\). Without loss of generality, we can assume that both \(\{i \in I \mid d \in X(i)\}\in F\) and \(\{i \in I \mid d \in Y(i)\}\in F\), since we can interchange \(X\) with \(\neg X\) and \(Y\) with \(\neg Y\) (if necessary).

Clearly, \(\neg X/F \neq \bar{\varnothing}/F\) and \(\neg Y/F \neq \bar{\varnothing}/F\), so \(\Diamond \neg X/F = \bar{W}/F= \Diamond \neg Y/F\), since \(\mathbf{S} \models  \forall x \colon x \neq 0 \to \Diamond x = 1\). We have \(\{i \in I \mid d \in X(i)\}\in F\) and \(\{i \in I \mid d \in Y(i)\}\in F\), so \(\{i \in I \mid d \notin {\sim} X(i)\}\in F\) and \(\{i \in I \mid d \notin {\sim} Y(i)\}\in F\). By Fig. \ref{fig1},  \(\{i \in I \mid c_1 \in {\sim} X(i)\}\in F\) and \(\{i \in I \mid c_1 \in {\sim} Y(i)\}\in F\). Thus, \(\{ i \in I \mid c_1, d \notin X(i) \cap {\sim} Y(i)\} \in F\) and \(\{ i \in I \mid c_1, d \notin {\sim} X(i) \cap Y(i)\} \in F\). By Fig. \ref{fig1}, \(\{ i \in I \mid d \notin \Diamond(X \wedge \neg Y)(i)\} \in F\) and \(\{i \in I \mid d\notin \Diamond(\neg X \wedge Y)(i)\} \in F\), hence \(\Diamond(X/F \wedge \neg Y/F) \neq \bar{W}/F\) and \(\Diamond(\neg X/F \wedge Y/F) \neq \bar{W}/F\). Since \(X/F\neq Y/F\) and \(X /F \neq \neg Y/F\), it follows that \( X/F\wedge \neg Y/F \neq \bar{\varnothing}/F\) or \(\neg X/F \wedge Y/F \neq \varnothing/F\), so this contradicts the fact that \(\mathbf{S} \models \forall x \colon x \neq 0 \to \Diamond x = 1\). Thus, we must have \(|S| \leqslant 4\). Since \(\mathbf{S} \models \forall x \colon x \neq 0 \to \Diamond x = 1\) and \(\mathbf{D}_N\) has no trivial subalgebras, this implies that \(\mathbf{S} \cong \textsf{Cm}(\mathbb{K}_1)\) or \(\mathbf{S} \cong \textsf{Cm}(\mathbb{K}_2)\), as claimed. \end{proof}

\begin{lemma}\label{clo}
Let \(N \subseteq \mathbb{E}\), let \(\gamma\) be the natural closure operator of \(\mathbf{D}_N\) and let \(X\) be a \(\gamma\)-closed element of \(D_N\) with \(\Diamond X \neq W\) and \(\Diamond^2 X = W\). Then \(\Diamond {\sim} X \neq W\).
\end{lemma}

\begin{proof}
Firstly, assume that \(a\in X\). Based on Fig. \ref{fig1}, we have \(a, b_3 \in \Diamond X\), hence \(a, b_3 \notin {\sim}\Diamond X\). So, by Lemma~\ref{natclo}(i), \(b_3 \notin \Diamond{\sim}\Diamond X = {\sim}X\), hence \(a,b_3 \in X\). By Fig. \ref{fig1}, \(b_3 \notin \Diamond{\sim} X\), which implies that \(\Diamond {\sim} X \neq W\) if \(a \in X\).

Now, assume that \(a \notin X\). We claim that \(a\in \Diamond X\); suppose that \(a \notin \Diamond X\). By Fig. \ref{fig1}, we have \(b_3 \notin X\), hence \(a,b_3\notin X\) and \(a \notin \Diamond X\). Thus, \(b_3 \notin \Diamond X\), which contradicts the fact that \(\Diamond^2 X = W\). It follows that \(a\in \Diamond X\), as claimed. By Fig. \ref{fig1}, we must have \(b_2\in X\) or \(c_2 \in X\), as \(a \notin X\) and \(\Diamond^2 X = W\). Hence, \(a,b_2,c_2 \in \Diamond X\), so by Lemma~\ref{natclo}(i), we have \(c_2 \notin \Diamond^2 {\sim}\Diamond X  = \Diamond{\sim} X\). From this, it follows that \(\Diamond {\sim} X \neq W\) if \(a \notin X\), so \(\Diamond{\sim} X \neq W\), as claimed. 
\end{proof}

\begin{lemma}\label{subD}
Let \(N \subseteq \mathbb{E}\), let \(F\) be an ultrafilter over a set \(I\), let \(\mathbf{S}\) be a subalgebra of \(\mathbf{D}_N^I /F\), let \(S\) be the universe of \(\mathbf{S}\), let \(\bar{X} \colon I \to D_N^I\) be def\-ined by \(i \mapsto X\), for each \(X \in D_N\), let \(\gamma\) be the natural closure operator of \(\mathbf{S}\) and let \(X \in D_N^I\) with \(X/F \in S\) and \(X/F \neq \bar{\varnothing}/F\).
\begin{enumerate}
\item
If \(\{ i \in I \mid P \cap X(i) = \varnothing\} \in F\), then \(\bar{D}/F \in S\).
\item
If \( \{i \in I \mid P \subseteq  \Diamond X(i) \} \in F\) and \(\Diamond X/F \neq \bar{W}/F\), then \(\bar{D} \in S\).
\item
If \(\{i \in I \mid P \subseteq \Diamond{\sim} X(i)\} \in F\), \(\Diamond X/F \neq \bar{W}/F\), \(\Diamond^2 X/F = \bar{W}/F\) and \(X/F\) is \(\gamma\)-closed, then \(\bar{D}/F \in S\).
\end{enumerate}
\end{lemma}

\begin{proof}
Assume that \(\{ i \in I \mid P  \cap X(i) = \varnothing\} \in F\). By Fig. \ref{fig1}, if \(Y \in D_N \setminus \{\varnothing\}\) with \(Y \cap P = \varnothing\), then \(\Diamond^2 Y = {\sim}D\), \(\Diamond^3 Y = {\sim}D\) or \(\Diamond^4 Y = {\sim}D\). Thus, \(\{i \in I \mid D = {\sim}\Diamond^2X(i)\} \cup \{ i \in I \mid D = {\sim}\Diamond^3X(i)\} \cup \{ i \in I \mid D= {\sim}\Diamond^4X(i)\} = I \in F\), so we must have \(\{i \in I \mid D = {\sim}\Diamond^2X(i)\} \in F\), \(\{i \in I \mid D = {\sim}\Diamond^3X(i)\} \in F\) or \(\{i \in I \mid D = {\sim}\Diamond^4 X(i)\} \in F\). Clearly, this implies that \(\neg \Diamond^2  X/F = \bar{D}/F\), \(\neg \Diamond^3 X/F = \bar{D}/F\) or \(\neg \Diamond ^4 X/F = \bar{D}/F\), so (i) holds.

Now, to prove (ii), assume that we have \( \{i \in I \mid P \subseteq  \Diamond X (i) \} \in F\) and \(\Diamond X/F \neq \bar{W}/F\). Then \(\{ i \in I \mid P \cap {\sim}\Diamond X(i) = \varnothing\}  \in F\), \(\neg \Diamond X/F \neq \bar{\varnothing}/F\) and \(\neg \Diamond X/F \in S\). By the previous result, \(\bar{D}/F \in S\), so (ii) holds.

To prove (iii), assume that \(\{i \in I \mid P \subseteq \Diamond{\sim} X(i)\} \in F\), \(\Diamond X/F \neq \bar{W}/F\), \(\Diamond^2 X/F = \bar{W}/F\) and \(X/F\) is \(\gamma\)-closed. From Lemma \ref{clo} and \L o\'s's Theorem, it follows that \(\Diamond \neg X/F \neq \bar{W}/F\). So, based on the previous result, \(\bar{D}/F \in S\). Thus, the three required results hold.
\end{proof}

\begin{lemma}\label{big}
Let \(N \subseteq \mathbb{E}\), let \(F\) be an ultraf\-ilter over a set \(I\) and let \(\mathbf{S}\) be a subalgebra of \(\mathbf{D}_N^I/F\). If \(\mathbf{S} \models \exists x \colon x \neq 0 \mathrel{\&} \Diamond x \neq 1\), then \(\mathbf{D}_N \in \mathop{\textup{IS}} (\mathbf{S})\).
\end{lemma}

\begin{proof}
Let \(S\) be the universe of \(\mathbf{S}\), let \(\gamma\) be the natural closure operator of \(\mathbf{S}\) and define a map \( \bar{X} \colon I \to D_N\) by \(i \mapsto X\), for each \(X \in D_N\). Since \(D\) generates \(\mathbf{D}_N\) and the natural diagonal map embeds \(\mathbf{D}_N\) into \(\mathbf{D}_N^I/F\), it will be enough to show that \(\bar{D}/F \in S\).

By Lemma \ref{nicegen}, there is some \(X \in D_N^I\) such that \(X/F \in S\), \(\Diamond X/F \neq \bar{W}/F\), \(\Diamond^2 X/F = \bar{W}/F\) and \(X/F\) is \(\gamma\)-closed, as \(\mathbf{S} \models \exists x \colon x \neq 0 \mathrel{\&} \Diamond x \neq 1\) and \(\mathbf{S} \models \forall x\colon  x \neq 0 \to \Diamond^5 x = 1\). If \(Y \subseteq W\), we either have \(c_2 \in Y\) or \(c_2 \in {\sim} Y\), so by Fig. \ref{fig1}, we must have \( P \subseteq \Diamond Y\) or \(P \subseteq \Diamond{\sim} Y\) if \(Y \subseteq W\). From this, it follows that \(\{i \in I \mid P  \subseteq \Diamond X(i)\} \cup \{i \in I \mid P  \subseteq \Diamond{\sim} X(i)\}= I \in F\), so \(\{i \in I \mid P \subseteq \Diamond X(i)\} \in F\) or \(\{i \in I \mid P \subseteq \Diamond{\sim} X(i)\} \in F\). By Lemma \ref{subD}, \(\bar{D}/F \in S\) and we are done.
\end{proof}

\begin{lemma}
Let \(N \subseteq \mathbb{E}\). Then \(\mathop{\textup{Var}} (\mathbf{D}_N)\) is
of height 3.
\end{lemma}

\begin{proof}
By Corollary~\ref{DN-inf} and Lemma~\ref{embed}, \(\mathbf{D}_N\) is infinite and \(\mathsf{Cm}(\mathbb{K}_2) \in \textup{IS}(\textbf{D}_N)\). By Lemma~\ref{simple}, each element of \(\textup{P}_\textrm{U}(\textbf{D}_N)\) is simple. By Lemmas~\ref{small} and \ref{big}, we must have \(\mathbf{B} \cong \textsf{Cm}(\mathbb{K}_1)\), \(\mathbf{B} \cong \textsf{Cm}(\mathbb{K}_2)\) or \(\mathbf{D}_N \in \textup{IS}(\mathbf{B})\) if \(\bf{B} \in \textup{ISP}_\textrm{U}(\textbf{D}_N)\). So, by Lemma~\ref{conditions}, \(\textup{Var}(\mathbf{D}_N)\) has height 3, as claimed. \end{proof}

Now it remains to show that for distinct  \(N,M\subseteq\mathbb{E}\), the
varieties \(\mathop{\textup{Var}} (\mathbf{D}_N)\) and
\(\mathop{\textup{Var}} (\mathbf{D}_M)\) are distinct. 

\begin{lemma}\label{ddd}
Let \(N\subseteq \mathbb{E}\) and let \(X \in D_N\setminus \{\varnothing\}\) with \(\Diamond_N^4 X \neq W\). Then \(X = D\) or \(\Diamond_N^4 X = {\sim} D\).
\end{lemma}

\begin{proof}
Based on Fig. \ref{fig1}, if \(a \in \Diamond X\) or \(c_2 \in X\), then we must have \(\Diamond^4 X = W\), hence \(X \subseteq U \cup C_2 \cup D\). By Fig. \ref{fig1}, \(\Diamond^4 D = W \setminus (C_2 \cup U) \neq W\). Similarly, \(\Diamond^4 X = W\) if \(d \in X\) and \(X \cap (C_2 \cup U) \neq \varnothing\), and \(\Diamond^4 X= {\sim}D\) if \(X \subseteq C_2 \cup U\). Since \(\Diamond^4 X \neq W\), we must have \(X = D\) or \(\Diamond^4 X = {\sim} D\), as claimed. 
\end{proof}

\begin{lemma}\label{pres}
Let \(M, N \subseteq \mathbb{E}\) and let \(u \colon \mathbf{D}_M \to \mathbf{D}_N\) be an embedding. Then \(u(D) = D\).
\end{lemma}

\begin{proof}
Suppose, for a contradiction, that \(u(D) \neq D\). Since \(u\) is an embedding, \(u(D) \neq\varnothing\). Based on Fig. \ref{fig1},  \(\Diamond_N^4 u(D) = u(\Diamond_M^4 D) = u(W \setminus (C_2 \cup U)) \neq W\), since \(u\) is an embedding. So, by Lemma \ref{ddd}, we must have \(\Diamond_N^4u(D) = {\sim} D\). By Lemma \ref{singletons}, we have \(U_1 \in D_M\) and \( (C_1 \cup U) \setminus U_1= (C_1 \cup U) \cap {\sim} U_1 \in D_M \). Now,
\[
u(U_1) \cup u((C_2 \cup U)\setminus U_1) = u(C_2 \cup U) = u ({\sim}\Diamond_M^4D) = {\sim}\Diamond_N^4u(D) = D,
\]
hence we must have \(u(U_1) = D = u((C_2 \cup U)\setminus U)\), \(u(U_1) = \varnothing = u(\varnothing)\) or \( u((C_2 \cup U) \setminus U_1) = \varnothing = u(\varnothing)\), which contradicts the fact that \(u\) is an embedding. Thus, \(u(D) = D\), as claimed.\end{proof} 

\begin{lemma}\label{diff}
Let \(M, N \subseteq \mathbb{E}\) with \(M \neq N\). Then \(\mathop{\textup{Var}} (\mathbf{D}_M) \neq \mathop{\textup{Var}} (\mathbf{D}_N)\).
\end{lemma}

\begin{proof}
Suppose, for a contradiction, that we have \(\mathop{\textup{Var}} (\mathbf{D}_M) = \mathop{\textup{Var}} (\mathbf{D}_N)\). By Lemmas~\ref{simple} and~\ref{big}, there are embeddings \(u \colon \mathbf{D}_M  \to \mathbf{D}_N\) and \(v \colon \mathbf{D}_N \to \mathbf{D}_M\). As \(M \neq N\), we have \((M \setminus N) \cup (N \setminus M) \neq \varnothing\). Let \(i := \min((M \setminus N) \cup (N \setminus M))\).  Without loss of generality, we can assume that \(i\in M\), since we can interchange \(M\) with \(N\) (if necessary). From the proof of Lemma~\ref{singletons}, there are unary terms \(t_A\), \(t_{L_i}\) and \(t_{U_{i-1}}\) with \(t_{A}^{\mathbf{D}_M}(D) = A = t_{A}^{\mathbf{D}_N}(D)\), \(t_{L_i}^{\mathbf{D}_M}(D) = L_i = t_{L_i}^{\mathbf{D}_N}(D)\) and \(t_{U_{i-1}}^{\mathbf{D}_M}(D) = U_{i-1} = t_{U_{i-1}}^{\mathbf{D}_N}(D)\), since \(i\) is the minimum of \((M \setminus N) \cup (N \setminus M)\). Now, let \(t(x)\) be the unary term defined by
\[
t(x) \deq \Diamond t_{L_i}(x) \wedge
\neg(t_A (x) \vee t_{L_i}(x) \vee t_{U_{i-1}}(x)).
\]
Based on Fig.~\ref{fig1} and Fig.~\ref{fig2}, we have \(t^{\mathbf{D}_M}(D) = U_i\) and \(t^{\mathbf{D}_N}(D) = U_i \cup
U_{i+1}\). Using Lemma \ref{ddd},  we find that
\[
v(U_i) \cup v(U_{i+1}) = v( t^{\mathbf{D}_N}(D)) = t^{\mathbf{D}_M}(v(D)) = t^{\mathbf{D}_M}(D) = U_i,
\]
so we have \(v(U_i) = U_i = v(U_{i+1})\), \(v(U_i) = \varnothing = v(\varnothing)\) or \(v(U_{i+1}) = \varnothing = v(\varnothing)\). This contradicts the injectivity of \(v\), so \( \mathop{\textup{Var}} (\mathbf{D}_M) \neq \mathop{\textup{Var}} (\mathbf{D}_N)\), as claimed.
\end{proof}

\section{Conclusion}

We have constructed a continuum of subvarieties of\/ \(\vKTB\) of height
3. Our main result follows immediately.

\begin{theorem}
The class of normal axiomatic extensions of \(\mathbf{KTB}\) of codimension \(3\)
is of size continuum.
\end{theorem}

It will be of interest to see what our result implies about 
subquasivarieties of \(\vKTB\) of small height, or, equivalently, 
about logics in \(\mathrm{CNExt}(\KTB)\) of small codimension
(see Remark~\ref{quasiv}). However, from Blanco, Campercholi and
Vaggione~\cite{BCV01} it follows that even the lattice of subquasivarieties of
\(\mathop{\textup{Var}}(\mathsf{Cm}(\mathbb{K}_2))\) is not a chain, so
the lattice of subquasivarieties of \(\mathop{\textup{Var}} (\mathbf{D}_N)\)
may be already quite complex, in particular, it may be of height strictly
greater than 3.

\section*{Acknowledgment}

This is a pre-print of a paper contributed to Advances in Modal Logic 2018. The final authenticated version is available online at:  \url{http://www.aiml.net/volumes/volume12/Koussas-Kowalski-Miyazaki-Stevens.pdf}.

%% Appendix.
%% Remove the \Appendix command if an 
%% appendix is not required.

%\Appendix
%Here starts the appendix. If you don't wish an appendix, please remove the \verb|\Appendix| command from the \LaTeX\ file.

%% Bibliography
%% Make sure to use the bibliographystyle aiml18.
\bibliographystyle{aiml18}

\end{document}